\documentclass{amsart}
\usepackage{amsfonts,amssymb,amscd,amsmath,enumerate,verbatim,newlfont,calc}
\usepackage{graphicx}

\newtheorem{theorem}{Theorem}[section]
\newtheorem*{theoremA}{Theorem A}

\newtheorem{lemma}[theorem]{Lemma}

\newtheorem{question}[theorem]{Question}
\newtheorem{claim}[theorem]{Claim}
\newtheorem{definition}[theorem]{Definition}
\newtheorem{corollary}[theorem]{Corollary}

\newtheorem{remark}[theorem]{Remark}

\newtheorem{proposition}[theorem]{Proposition}
\begin{document}
\author{A.Ehsani, F.H.Ghane and M.Zaj}
\address{Department of Mathematics, Ferdowsi University of Mashhad, Mashhad, Iran}
\email{azam.ehsani65@stu.um.ac.ir}
\address{Department of Mathematics, Ferdowsi University of Mashhad, Mashhad, Iran}
\email{ghane@math.um.ac.ir}
\address{Department of Mathematics, Ferdowsi University of Mashhad, Mashhad, Iran}
\email{marzie.zaj@stu.um.ac.ir}
\title[on ergodicity of mostly expanding semi-group actions ] { on ergodicity of mostly expanding semi-group actions }
\maketitle
\begin{abstract}
In this work, we address ergodicity of smooth actions of finitely generated semigroups on
an $m$-dimensional closed manifold $M$. We provide sufficient conditions for such an action to be ergodic
with respect to the Lebesgue measure. Our results improve the main result in \cite{DKN}, where the ergodicity
for one dimensional fiber was proved. We will introduce Markov partition for finitely generated semi-group actions and then we establish ergodicity for a large class of finitely generated semi-groups of $C^{1+\alpha}$-diffeomorphisms that admit a Markov partition.

Moreover, we present some transitivity criteria for semi-group actions and provide a weaker form of dynamical irreducibility
 that suffices to ergodicity in our setting.
\end{abstract}
\section{introduction}\label{int}
One of the main goals in dynamical systems is to describe the typical behavior of orbits and find that how they evolve in time.
An interesting approach in this direction is given by ergodic theory, which describes the typical behavior of orbits from a measurable point of view. In this context ergodic measures play a key role in understanding the ergodic behavior of a dynamical system.

In this paper, we discuss ergodic properties of random iterations of a finite family of $C^{1+\alpha}$-diffeomorphisms defined on an $m$-dimensional compact manifold.

Some knowledge of ergodic properties of semi-group or group action of diffeomorphisms is already available, through works of several authors.
For instance, see \cite{S,DKN} for the case that the ambient space is $S^{1}$ and \cite{BFS} for general case.\\
Motivated by these results, we consider semi-groups or groups of $C^{1+\alpha}$-diffeomorphisms on a closed manifold $M$ exhibiting mostly expanding behavior and establish their ergodicity with respect to Lebesgue measure.

Write $\mathcal{G}^+$ (or $\mathcal{G}$) for the semigroup (or group) generated by a finite family of diffeomorphisms on a compact manifold $M$.

A subset $B\subset M$ is \emph{forward invariant} for $\mathcal{G}^+$ , if $f(B)\subset B$ for all
$f\in \mathcal{G}^+$.
For the group action case, we can replace it by $f(B)= B$ for all $f\in \mathcal{G}$.

An invariant probability measure $\mu$ for a semigroup (or group) action is \emph{ergodic} if for every measurable forward $\mathcal{G}^+$-invariant (or $\mathcal{G}$-invariant) subset $A\subset M$ either $\mu(A)=0$ or $\mu(M\setminus A)=0$. This definition can be extended to a measure $\mu$, which is \emph{quasi-invariant}, i.e. $f_{*}\mu$ is absolutely continuous
 with respect to $\mu$ for each element $f$ in the acting semigroup or group.
Note that for any $C^1$-map, the Lebesgue measure is quasi-invariant \cite{DKN}.

In this context, the following question is natural \cite{DKN}:
\begin{question}
 Under what conditions  a smooth action of a semigroup or group on a compact manifold is Lebesgue ergodic?
\end{question}
The main results of \cite{S} and \cite{DKN}, solve the question in the affirmative for the one dimensional case under some additional assumptions. Indeed, for finitely generated groups of $C^{1+\alpha}$ diffeomorphisms on $S^{1}$, if $\mathcal{G}$ acts minimally and its Lyapunov expansion exponent is positive then the action is ergodic. \\
We recall that the action of a semigroup (or group) is \emph{minimal} if every closed invariant subset is either empty or coincides with the whole space, equivalently total orbit of each point is dense in the ambient space.

A point $x\in M$ is said to be non-expandable provided that for every $g\in \mathcal{G}^+$ (or $g \in \mathcal{G}$) one has $m(Dg(x))\leq1$, otherwise is called an expandable point.

In \cite{DKN}, the authors proved that
if $\mathcal{G}$ is a finitely generated subgroup of $Diff^{2}(S^{1})$
such that it acts minimally and possesses a finite number of non-expandable points with some aditional assumption, then the action of $\mathcal{G}$ is ergodic with respect to the Lebesgue measure.

Several authors also worked on ergodicity of semi-group actions on compact $m$-dimensional manifolds. As shown in \cite{BFS}, ergodicity happens whenever the semi-group action is expanding minimal $C^{1+\alpha}$-conformal.\\
 The concept of expanding action in \cite{BFS} has different means: a semi-group action is \emph{expanding} provided that its inverse semigroup behaves locally expandable.

 In general, minimality \cite{F} does not imply ergodicity.
In the opposite direction ergodicity does not imply minimality. Indeed, one can easily construct examples of ergodic group actions having a global
fixed points.\\
Moreover, the main result of this paper illustrates an example of an ergodic semi-group action for which minimality is not a necessary condition.

 Here, we improve the results of \cite{S,DKN} for the case that the ambient space is $m$-dimensional, $m\geqslant 2$.
 In our setting, we allow that the non-expandable points exist and they are contained in an open region $V\subset M$.

 In fact, we contribute sufficient conditions for ergodicity that applicable to a large class of finitely generated semi-groups of $C^{1+\alpha}$-diffeomorphisms.\\
 Our hypothesis that formulated in section 2 are conditions of the type that Oliveira and Viana \cite{OV} introduced for deterministic non-uniformly expanding maps. We extend the approach used by them to semi-groups of $C^{1+\alpha}$-diffeomorphisms on a compact $m$-dimensional manifold. However, this extension involves some difficulties. For instance, Markov partitions are not known to exist for semi-group actions.

Let us note that non-uniformly expanding maps with some random noise was addressed in several works. In particular, the stochastical stability of general classes of non-uniformly expanding maps were established in \cite{AA}, \cite{AV}, \cite{Y} and \cite{BY}. Also in \cite{BFS}, the authors introduced a strong form of non-uniform expanding property for random maps. This concept in our setting has slightly different means.
 Indeed, we will establish the ergodicity of semi-group actions which equip with a weak form of non-uniform expanding property. This property is referred to as \emph{orbital non-uniform expanding property}.

 In section 6, some criteria for dynamical irreducibility is obtained.
 It provides a weak form of irreducibility in comparing with minimality. Also, the concept of weak cycle property will be introduced. This concept is equivalent to density of backward orbits for a subset of points with full Lebesgue measure and suffices to establish ergodicity from local ergodicity.
 Finally, we construct some examples that fit our assumptions in the last section.
 \section{\bf{Preliminaries and results}}\label{non}
Suppose that $X$ is a compact metric space and $\mathcal{F}=\{f_1,\cdots,f_k\}$ is a finite family of homeomorphisms of $X$.
Let us consider the \emph{symbol space} $\Sigma_k^+$
which is the set of one sided infinite words over the alphabet $\{1,\ldots, k\}$.
For any sequence $\omega=(\omega_0\omega_1\cdots \omega_n\cdots)\in\Sigma_k^+$, we write $f^n_\omega(x)=f_{\omega_{n-1}}\circ f_\omega^{n-1}(x)$; $n\in\mathbb{N}$ and $f_\omega^0=Id$.
A fiber orbit corresponding to a one-way infinite word $\omega=(\omega_0\omega_1\cdots)\in\Sigma_k^+$ at a point $x$ is defined by $\mathcal{O}^+(x,\omega)=\{f^n_\omega(x)\}_{n=0}^\infty$.
\begin{definition}
We say that a finite family $\mathcal{R}=\{R_1,\cdots,R_k\}$ of subsets of $X$ is a topological partition if it satisfies the following:
 \\
 (1) each $R_i$ is open in $X$;
 \\
 (2) $R_i\cap R_j=\emptyset$, for each $i,j=1,\cdots,k$, with $i\neq j$;
 \\
 (3) $X=\overline{R_1}\cup\cdots\cup \overline{R_k}$.
 \end{definition}
 In below, we generalize the concept of Markov partition to finitely semi-group actions.

 Take a topological partition $\mathcal{R}=\{R_1,\cdots,R_k\}$ of a compact metric space $X$ with together a finite family $\mathcal{F}=\{f_1,\cdots,f_k\}$ of homeomorphisms of $X$.
 Let us consider the restriction $f_i\vert_{R_i}$, for each $i=1,\cdots,k$.
\\
We say that $(\mathcal{R},\mathcal{F})$ satisfies the $(n,\omega)$-fold intersection property for a positive integer $n\geq 3$ and $\omega=( \omega_0\omega_1\omega_2\cdots)\in\Sigma_k^+$ if $R_{\omega_j}\cap f^{-1}_{\omega_j}(R_{\omega_{j+1}})\neq\emptyset$, $0\leq j\leq n-1$, implies that $\cap_{j=0}^{n-1}(f_\omega^j)^{-1}(R_{\omega_{j+1}})\neq\emptyset$.
\\
A topological partition $\mathcal{R}=\{R_1,\cdots,R_k\}$ is \emph{Markov} for semigroup generated by $\mathcal{F}=\{f_1,\cdots,f_k\}\subset home(X)$, if $(\mathcal{R},\mathcal{F})$ satisfies the $(n,\omega)$-fold intersection property for all $n\geq 3$ and $\omega\in\Sigma_k^+$.

It is not hard to see that $(\mathcal{R},\mathcal{F})$ is Markov provided that it
 satisfies the following condition:
 if $f_i(R_i)\cap R_j\neq\emptyset$, then $f_i(R_i)\supset R_j$, for each $i,j=1,\cdots,k$.

\subsection{\bf{Statement of the main result}}
\label{non}
In the rest of the paper we take M to be a compact Riemannian manifold, $m$ denotes the normalized Lebesgue measure
 and $\{f_1, \ldots, f_p, f_{p+1}, \ldots, f_{p+q} \}$
a finite family of $C^{1+\alpha}$-diffeomorphisms defined on $M$ satisfying the following conditions:

$(A_0)$ There exists a topological partition $\mathcal{R}=\{R_1,\cdots, R_p, R_{p+1}, \ldots, R_{p+q}\}$ such that for each $i=1,\cdots,p+q$, the clouser $\overline{R_i}$ has finite inner diameter, i.e. any two points in $\overline{R_i}$ may be joined by a curve contained in $\overline{R_i}$ whose length is bounded by a constant $L$.  Moreover, $m(\partial R_i)=0$, for each $i=1,\cdots,p+q$.

$(A_1)$ If $f_i(R_i)\cap R_j\neq\emptyset$ then $f_i(R_i)\supset R_j$ which implies that $f_i(\overline{R_i})\supset \overline{R_j}$.

$(A_2)$ There exists $\sigma_1, \sigma_2>1$ such that
\begin{itemize}
\item $\Vert Df_i(x)^{-1}\Vert^{-1}\geq\sigma_1$ for every $x \in R_i$, $1 \leq i \leq p$,
\item  $\Vert Df_{p+j}(x)^{-1}\Vert^{-1}\geq\sigma_2^{-1}$ for every $x\in R_{p+j}$, $1 \leq j \leq q$
\end{itemize}
and $\sigma_2$ is close enough to 1.

$(A_3)$ $|det(Df_{p+j}(x))|> q$ for every $x \in R_{p+j}$, $1 \leq j \leq q$.

In fact, the acting of semi-group $\mathcal{G}^+$ generated by $\{f_1, \ldots, f_{p+q}\}$ is expanding on $R_1 \cup \ldots \cup R_p$ and is never very contacting.
According to condition $(A_1)$, the topological partition $\mathcal{R}$ is Markov for $\mathcal{F}$.
Also, by $(A_0)$, $m(\partial \mathcal{R})=0$ and therefore,  $m(\mathcal{O}_\mathcal{G}^+(\partial \mathcal{R}))=0$, where $\mathcal{O}_\mathcal{G}^+(\partial \mathcal{R})$ is the total orbit of $\partial \mathcal{R}$, i.e.
$$\mathcal{O}_\mathcal{G}^+(\partial \mathcal{R}) = \bigcup_{g \in \mathcal{G}^+} g(\partial \mathcal{R}).$$
Moreover, we take $N=M\backslash\mathcal{O}_\mathcal{G}^+(\partial R)$. Then $m(N)=1$.
\begin{definition}
We say that a semi-group $\mathcal{G}^+$
 satisfies the weak cycle property if for each open set $B\subset M$, there exists a sequence  $\{h_{i}\}\subset \mathcal{G}^{+}$ such that $M\circeq \bigcup_{i=1}^{\infty}h_{i}(B)$, this means that $m (M\setminus \bigcup_{i=1}^{\infty} h_{i}(B)=0)$.
\end{definition}
Now, we state the main result of the paper that illustrates the Lebesgue ergodicity for a large class of finitely generated semigroups of $C^{1+\alpha}$-diffeomorphisms. Here the concept of ergodicity is referred to ergodicity of quasi-invariant measures and our focus is on the Lebesgue measure since it is quasi invariant for $C^1$ maps.
\begin{theoremA}\label{main}
 Suppose that $\mathcal{G}^+$ is a finitely generated semi-group of $Diff^{1+\alpha}(M)$ for which the following holds:

 $(1)$ $\mathcal{G}^+$ satisfies the weak-cycle property;

 $(2)$ there exists a finite family $\{f_1, \ldots, f_{p+q}\}\subset \mathcal{G}^+$ satisfies conditions $(A0)$, $(A1)$, $(A2)$ and $(A3)$ above.

  Then $\mathcal{G}^+$ is ergodic with respect to the Lebesgue measure.
\end{theoremA}
\section{\bf{Non-uniform expanding semigroups}}\label{non}
Non-uniformly expanding maps introduced in \cite{A, Vi}. Then general conclusions for systems exhibiting non-uniform expanding behavior provided in \cite{ABV}.
Let us recall that a local diffeomorphism $f:M \to M$ is \emph{non-uniformly expanding} if there exists $c>0$ such that for Lebesgue almost every point $x\in M$ one has:
$$\limsup_{n \to \infty} \frac{1}{n}\sum_{j=0}^{n-1} \log \parallel Df(f^{j}(x))^{-1}\parallel <-c.$$
This approach has been most effective in studying ergodic aspects of systems and extended to the random setting \cite{V}: the authors considered random perturbations at each iterate a map, that is close to a non-uniformly expanding map, chosen independently according to some probabilistic law $\theta_{\varepsilon}$, where $\varepsilon >0$ is the noise level.

The notion of non-uniformly expanding on random orbits was addressed in \cite{BFS} with slightly different means.

To be more precise, consider a measure preserving system $(\Omega ,\sigma , \mathbb{P})$, where $\mathbb{P}$ is a Borel measure, $\sigma:\omega \to \omega$ is $\mathbb{P}$-invariant ergodic transformation and $\Omega$ is a compact separable metric space with a random continuous map $F:\Omega \rightarrow C^{r}(M,M)$.
We denote $f_{\omega}:= F(\omega)$, $f_{\omega}^n:= f_{\sigma^{n-1}\omega}\circ \dots \circ f_{\omega}$ and $f_{\omega}^{0}:= id$.\\
Now, given $x\in M$ and $\omega \in \Omega^{\mathbb{Z}}$, the sequence $(f^{n}_{\omega}(x))_{n\geq 1}$ is called a \emph{random orbit} of $x$.

We say that $F$ is a \emph{non-uniformly expanding on random orbits} if there exists $c>0$ such that for $(\mathbb{P}\times m)$-almost $(\omega,x)\in \Omega \times m$ it holds:
$$\limsup_{n \to \infty} \frac{1}{n}\sum_{i=0}^{\infty} \parallel Df_{\sigma^i\omega}(f^{j}_{\omega}(x)^{-1})\parallel<-c.$$
They proved that [Proposition 4.4, \cite{BFS}] there are no non-uniform expanding finitely generated semi-group of diffeomorphisms which acts on random orbits.


In this paper, we deal with random product of finitely many maps and semi-group action generated by these maps and then introduce a weak form of non-uniform expanding property that we addressed here.
\begin{definition}
Consider a finitely generated semigroup $\mathcal{G}^+$ (or group $\mathcal{G}$) with generators $\{f_0, \ldots, f_{k-1}\}$.
 We say that $\mathcal{G}^+$ (or $\mathcal{G}$) is orbital non-uniformly expanding if there exists $c>0$ such that for Lebesgue almost every $x\in M$, there is a sequence $\omega \in \Sigma_+^k$ (or $\omega \in \Sigma^k$) satisfying
  \begin{equation}\label{1}
\limsup_{n \to \infty} \dfrac{1}{n}\sum_{i=0}^{n-1}\log \parallel Df_{\omega_i}(f_\omega^i(x))^{-1}\Vert \leq -c.
 \end{equation}
\end{definition}
In the rest of this section, we show that a large class of finitely generated semi-groups of $C^{1+\alpha}$ diffeomorphisms on a compact manifold $M$ exhibit orbital non-uniformly expanding property.
 \begin{definition}
  Let $x\in N= M - \partial \mathcal{R}$ be given.
  Suppose that $x\in R_{\omega_0}$ and
 the indices $\omega_0, \omega_1, \cdots, \omega_{j-1} \in\{1, \cdots, p+q\}$ are chosen such that for each $k=0,...,j-2$,
  \begin{center}
  $f_{\omega_k} \circ \cdots \circ f_{\omega_0}(x)\in R_{\omega_{k+1}}.$
  \end{center}
 Now, we choose an index $\omega_j\in\{1, \cdots, p+q\}$ satisfying
 \begin{center}
 $f_{\omega_{j-1}}\circ \cdots \circ f_{\omega_0}\in R_{\omega_j}.$
 \end{center}
 In this way, inductively we may associate a sequence $\omega=\omega(x)$, to each $x\in N$, by taking $\omega=(\omega_0,\omega_1,\cdots)$ which is called the \emph{itinerary} of $x$. Let $\pi_n$ denote the projection that maps the sequence $\omega=(\omega_0,\omega_1,\cdots)$ to the finite word $\omega(x, n)=(\omega_0,\omega_1,\cdots, \omega_{n-1})$ which we refer to it as $n$-\emph{itinerary} of $x$.
 \end{definition}
 \begin{proposition}\label{uniform}
 Consider semi group $\mathcal{G}^+ \subset Diff^{1+\alpha}(M)$ for which there exist a finite family $\{f_1, \ldots, f_p, f_{p+1}, \ldots, f_{p+q} \}\subset\mathcal{G}^+$ and a topological partition
$\mathcal{R}=\{R_1,\cdots, R_p, R_{p+1}, \ldots, R_{p+q}\}$ satisfying conditions $(A0)$, $(A1)$, $(A2)$ and $(A3)$. Then $\mathcal{G}^+$ is orbital non-uniformly expanding.
 \end{proposition}
\begin{proof}
Consider the subset
$\{f_1, \ldots, f_p, \ldots, f_{p+q}\}$ of $\mathcal{G}^+$ and a family $\mathcal{R}$ of open subsets of $M$ satisfying the assumptions.
Also, for each $x\in N$, take a sequence $\omega=\omega(x)$
which is the itinerary of $x$.
Since $N$ has full Lebesgue measure and condition $(A3)$ ensures that the approach applied by Alves, Bonatti and Viana [Lemma A1, \cite{ABV}] can be extended to our setting, hence the following claim holds.
\begin{claim} For Lebesgue almost every point $x\in M$, the $\omega$-orbital branch of $x$ spends a fraction $\epsilon_0>0$ of the time in $R_1\cup\cdots\cup R_p$, for some  $\epsilon_0>0$; that is
\begin{center}
$\# \{0\leq k< n: f_{\omega}^k(x)\in R_1\cup\cdots\cup R_p\}\geq\epsilon_0n$
\end{center}
for large enough $n$, where $\omega$ is the itinerary of $x$.
\end{claim}
Now, we use the claim to show that $\mathcal{G}^+$ is orbital non-uniformly expanding, that is the equality (\ref{1}) holds for Lebesgue almost every point $x\in M$.
Take $\epsilon_0>0$ as introduced in the claim and $\sigma_2$ close enough to 1 so that $\sigma_1^{- \epsilon_0}\sigma_2 \leq e^{-c}$, for some $c>0$. Let $x$ be any point satisfies the conclusion of the claim and $\omega=\omega(x)$ is the itinerary of $x$. Then
\begin{align*}
\prod_{j=0}^{n-1}\Vert Df_{\omega_j}(f_\omega^j(x))^{-1}\Vert\leq\sigma_1^{- \epsilon_0 n} \sigma_2^{(1- \epsilon_0)n}\leq e^{-c n}
\end{align*}
for every large enough $n$. This means that $x$ satisfies the conclusion of the proposition.
\end{proof}
 \section{Hyperbolic times and hyperbolic cyliders}\label{result}
 Throughout this section, we take $\{f_1, \ldots, f_p, f_{p+1}, \ldots, f_{p+q} \}\subset \mathcal{G}^+$
a finite family of $C^{1+\alpha}$-diffeomorphisms and a topological partition $\mathcal{R}=\{R_1,\cdots, R_p, R_{p+1}, \ldots, R_{p+q}\}$ satisfying the conditions $(A0)$, $(A1)$, $(A2)$, and $(A3)$.

 First, we present the concept of \emph{hyperbolic time} of a point $x\in M$ for semigroup action $\mathcal{G}^+$.
 This concept was introduced in \cite{ABV} for differentiable deterministic maps.
\begin{definition}
 Given $0<c<1$, we say that $n\in\mathbb{N}$ is a hyperbolic time for $x\in M$ if for every $1\leq k\leq n$, one has
\begin{equation}\label{2}
\prod_{j=n-k}^{n-1}\Vert Df_{\omega_j}(f_\omega^j(x))^{-1}\Vert\leq e^{-c k},
\end{equation}
\end{definition}
where $\omega=\omega(x)=(\omega_0, \omega_1, \ldots, \omega_j, \ldots)$ is the itinerary of $x$.\\
Any nonempty set of the form
$$C^n=C^n[\omega_0, \ldots, \omega_{n-1}]= \{x \in M : x \in R_{\omega_0}, f_{\omega_0}(x) \in R_{\omega_1}, \ldots, f_\omega^n(x) \in R_{\omega_{n-1}}\}$$
 is called a \emph{cylinder} of length $n$. We say that $C^n$ is a \emph{hyperbolic
cylinder} provided that $n$ is a hyperbolic time for each point $x \in C^n$ and its itinerary $\omega$.

Let $\mathcal{C}^n$ be the family of all cylinders of length $n$ and $\mathcal{C}^n_h$ the subset of hyperbolic cylinders.

\begin{remark}
The closure of $\mathcal{C}^{n}[\omega_{0},\dots ,\omega_{n-1}]$ is defined as
$$\bar{C}^{n}[\omega_{0},\dots ,\omega_{n-1}]=
\{y\in M : y\in \bar{R}_{\omega_{0}}, f_{\omega
 _{0}}(y)\in \bar{R}_{\omega_{1}}, \dots ,  f_{\omega}^n(y)\in \bar{R}_{\omega_{n-1}} \}.$$
  Hence, $f_{\omega_{j-1}}\circ \dots \circ  f_{\omega_{0}}(\bar{C}^{n}[\omega_{0},\dots , \omega_{n-1}])=\bar{C}^{n-j}[\omega_{j},\dots , \omega_{n-1}]$, for any $1\leq j<n$.\\
So, $f_{\omega_{n-1}}\circ \dots \circ f_{\omega_{0}}(\bar{C}^{n}[\omega_{0},\dots ,\omega_{n-1}])=f_{\omega_{n-1}}(\bar{R}_{\omega_{n-1}})$ and thus its inner diameter is bounded by the constant $$K_{2}=K_{1} max_{1\leq i\leq p+q} \{\parallel Df_{i}(x)\parallel : x\in R_{i}\},$$ where $K_{1}$ is the maximum inner diameter of $\bar{R}_{i}$ over all $i=1,\dots ,p+q$.\\
\end{remark}
\begin{definition}
Let $x\in M$ with itinerary $\omega$. We say that the frequency of hyperbolic times for $x$ is greater than $\epsilon_0>0$ if for large $n\in\mathbb{N}$, there are $l\geq\epsilon_0 n$ and integers $1\leq n_1<\cdots<n_l\leq n$ which are hyperbolic times for $x$.
\end{definition}
\begin{remark}
Let us note that if $\omega$ is the itinerary of $x$ then $\sigma^k\omega$ is the itinerary of $f_\omega^k(x)$, where $\sigma$ is the Bernoulli shift transformation.\\
If $n$ is a hyperbolic time for $x$ with itinerary $\omega$, then $n-s$ is a hyperbolic time for $f^s_\omega(x)$, for any $1\leq s<n$.
\\ It is not hard to see that the converse is also true. Indeed, if $k<n$ is a hyperbolic time for $x$ and there exists $1\leq s\leq k$ such that $n-s$ is a hyperbolic time for $f^s_\omega(x)$ then $n$ is a hyperbolic time for $x$.
\end{remark}
The next lemma asserts that for points satisfy the orbital non-uniformly expanding property (\ref{1}), there are infinitely many hyperbolic times. Moreover, the set of hyperbolic times has positive density at infinity, and its proof is based on a lemma due to Pliss (see e.g. \cite{M}).
\begin{lemma}\label{frequency}
Suppose that $x\in M$, $\omega$ is its itinerary and $n\geq 1$ is such that
\begin{equation}\label{3}
\frac{1}{n}\sum_{j=0}^{n-1}\log\Vert Df_{\omega_j}(f_{\omega}^j(x))^{-1}\Vert\leq -c<0.
\end{equation}
Then, there is $\epsilon_0>0$, depending only on $\mathcal{G}^+$ and $c$, and a sequence of hyperbolic times $1\leq n_1<\cdots<n_t\leq n$ for $x$, with $t\geq\epsilon_0 n$; that is the frequency of hyperbolic times for $x$ is larger than $\epsilon_0$.
\end{lemma}
\begin{proof}
Similar to Corollary 3.2 of \cite{ABV}. See also Proposition 4.4 of \cite{V}.
\end{proof}
Suppose that $x$ satisfies the conclusion of Proposition \ref{uniform}, so for large enough $n>0$
\begin{align*}
\sum_{j=0}^{n-1}\log\Vert Df_{\omega_j}(f_\omega^j(x))^{-1}\Vert<-nc,
\end{align*}
thus
\begin{align*}
\prod_{j=0}^{n-1}\Vert Df_{\omega_j}(f_\omega^j(x))^{-1}\Vert<e^{-nc}.
\end{align*}
 Then $n$ is a hyperbolic time for $x$.
\\ Let $\mathcal{H}$ be the set of all points $x\in M$ with itinerary $\omega$ satisfying the orbital non-uniformly expanding property (\ref{1}). Clearly $\mathcal{H}$ has full Lebesgue measure.

Here, the ball of radius $r>0$ is meant with respect to the Riemannian distance $dist(x, y)$ on $M$. If $D \subset M$
is any path connected domain, we define the distance $dist_D(x, y)$ between two points $x$ and $y$ in $D$
to be the infimum the lengths of all curves joining $x$ to $y$ inside $D$. In particular, $dist_D(x, y)\geq dist(x, y)$ for every $x$ and $y$ in $D$.
\begin{lemma}
There is $r>0$ such that if $n$ is a hyperbolic time for $x\in\mathcal{H}$ with itinerary $\omega=(\omega_0, \omega_1, \ldots)$, then
\begin{equation}\label{5}
\|Df_{\omega_0}(y)^{-1}\| \leq e^{\frac{c}{2}} \|Df_{\omega_0}(x)^{-1}\|,
\end{equation}
for any point $y$ in the ball $B(x, r e^{\frac{-nc}{2}})$.
\end{lemma}
\begin{proof}
By continuity we can choose $r > 0$ small enough so that the uniform bounds $\sigma_1$ and $\sigma_2$ in conditions $(A1)$ and $(A2)$ hold for the mappings $f_i$ in $r$-neighborhoods of $R_i$, $i=1, \ldots, p+q$.
Moreover, we take $r$ small enough so that the inverse
of the exponential map $exp_x$ is defined on the $r$-neighborhood of every point $x \in M$ and it is isometry on $B(x, r)$.

Now, by continuity and compactness of $M$ we can take $r > 0$ small enough such that if $n$ is a hyperbolic time for $x$ with itinerary $\omega= (\omega_0, \omega_1, \ldots)$ then
$$\|Df_{\omega_0}(y)^{-1}\| \leq e^{\frac{c}{2}} \|Df_{\omega_0}(x)^{-1}\|,$$ whenever $y \in B(x, r e^{\frac{-c}{2}}).$ Since $B(x, r e^{\frac{-nc}{2}}) \subset B(x, r e^{\frac{-c}{2}})$, the statement of the lemma follows by the above inequality.
\end{proof}
A \emph{dynamical ball} of center $x$, itinerary $\omega \in \Sigma^k_+$, radius $r$, and length $n \geq 1$ is defined by
$$B(\omega, x, n, r)=\{y \in M: dist(f_\omega^j(x), f_\omega^j(y))\leq r, \forall \ 0\leq j \leq n\}.$$
\begin{proposition}\label{preball}
There exists $r>0$ such that if $n$ is a hyperbolic time for $x\in\mathcal{H}$, then there exists a neighborhood $V(x, n)$ of $x$ for which the following hold:
\\1) $f_\omega^n$ maps $V(x, n)$ diffeomorphically onto $B(f_\omega^n(x),r)$, where $B(f_\omega^n(x),r)$ is a ball of radius $r$ and with center $f^n_\omega(x)$;
 \\ 2) for every $y\in V(x, n)$ and $1\leq j\leq n$, we have
 \begin{center}
 $\Vert D(f^j_{\sigma_{\omega}^{n-j}})^{-1}(z)\| \leq e^ {\frac{-jc}{2}},$
 \end{center}
 where $z$ belongs to $B(f_\omega^n(x),r)$.
\end{proposition}
\begin{proof}
Analogous to the proof of Proposition 4.9 of \cite{V} and according to Lemma 4.6.
\end{proof}
It is not hard to see that the neighborhood $V(x, n)$ is a dynamical ball about $x$.
\begin{corollary}
For every
 $1\leq j\leq n$ and $x,y$ in the closure of any $C^{n}=C^{n}[\omega_{0},\dots ,\omega_{n-1}]\in \mathcal{C}^{n}_{h}$ we have:
 $$d_{f^{n-j}_{\omega}(\bar{C}^{n})}(f^{n-j}_{\omega}(x),f^{n-j}_{\omega}(y))\leq e^{\frac{-jc}{2}}d_{f_\omega ^{n}(\bar{C}^{n})}(f^{n}_{\omega}(x),f^{n}_{\omega}(y))\leq K_{2}e^{\frac{-jc}{2}},$$
 where $f_\omega ^j= f_{\omega_{j-1}}\circ \ldots \circ f_{\omega_0}$.
\end{corollary}
\begin{proof}
Any curve joining $f^{n}_{\omega}(x)$ to $f^{n}_{\omega}(y)$ inside $f^{n}_{\omega}(\bar{C}^{n})$ lifts to a
unique curve joining $f^{n-j}_{\omega}(x)$ to $f^{n-j}_{\omega}(y)$ inside $f^{n-j}_{\omega}(\bar{C}^{n})$. Now we conclude the result by the above proposition.
\end{proof}
\begin{lemma}
There is $L_1>0$ such that for any $x,y$ in closure of any $C^{n}=C^n[\omega_0, \ldots, \omega_{n-1}] \in \mathcal{C}^{n}_{h}$
$$L^{-1}_{1}\leq \frac{\mid det Df^{n}_{\omega}(x)\mid}{\mid det Df^{n}_{\omega}(y)\mid}\leq L_{1}.$$
\end{lemma}
\begin{proof}
Note that by assumption $g_{j}=log \mid det Df_{j{\mid \bar{R}_{j}}}\mid$, for $j=1,\dots ,p+q$, is $\alpha$-H\"{o}lder and thus for each $ x,y \in R_{j}$, it holds that
$$\mid g_{j}(x)-g_{j}(y)\mid \leq C_0 d(x,y)^{\alpha},$$ for some constants $C_0> 0$ and $\alpha>0$.
Now, suppose that $C^{n}= C^n[\omega_0, \ldots, \omega_{n-1}]$ is a hyperbolic cylinder and $x, y \in \overline{C}^n$. This means that $x, y \in R_{\omega_0}, \ f_{\omega_0}(x), f_{\omega_0}(y) \in R_{\omega_1}$ and for each $1 \leq i \leq n$, $f_\omega ^i(x), f_\omega ^i(y) \in R_{\omega_i}$. According to Corollary 4.8, for each $1 \leq j \leq n$ and $x, y \in \overline{C}^n$ it holds that
$$d_{f^{n-j}_{\omega}(\bar{C}^{n})}(f^{n-j}_{\omega}(x),f^{n-j}_{\omega}(y))\leq e^{\frac{-jc}{2}}d_{f_\omega^{n}(\bar{C}^{n})}(f^{n}_{\omega}(x),f^{n}_{\omega}(y))\leq K_{2}e^{\frac{-jc}{2}},$$
and therefore,
$$log \frac{\mid det Df^{n}_{\omega}(x)\mid}{\mid det Df^{n}_{\omega}(y)\mid}=\sum^{n-1}_{i=0}\mid g_{\omega_{i}}(f^{i}_{\omega}(x))-g_{\omega_{i}}(f^{i}_{\omega}(y))\mid \leq \sum^{n-1}_{i=0}C_0d_{f_\omega^j(\overline{C}^n)}(f^{i}_{\omega}(x),f^{i}_{\omega}(y))^{\alpha}$$
$$\leq C_0 K_2^\alpha \sum^{n-1}_{i=0}e^{\frac{-(n-i)c\alpha}{2}} \leq C_0 K_2^\alpha \sum^{\infty}_{i=0} e^{\frac{-ic \alpha}{2}}.$$
Take $L_{1}=\exp(C_0 K_2^\alpha \sum^{\infty}_{i=0} e^{\frac{-ic \alpha}{2}}).$
\end{proof}
\begin{corollary}
There exists a constant $L_2>0$ such that if $C^{n}=C^n[\omega_0, \ldots, \omega_{n-1}]$ is a hyperbolic cylinder and $A_1, A_2$ are two subsets of $\overline{C}^n$ then
 \begin{align}\label{10}
 L_2^{-1} \dfrac{m(A_1)}{m(A_2)}\leq\dfrac{m(f_\omega^n(A_1))}{m(f_\omega^n(A_2))}\leq L_2\dfrac{m(A_1)}{m(A_2)}
 \end{align}
\end{corollary}
\begin{proof}
We apply the change of variable formula for $f_\omega^n$ and analogous to Corollaries 4.10, 4.11 and 4.12 of \cite{V}, we conclude the result.
\end{proof}
\section{\bf{Proof of Theorem A}}\label{non}
Let us take a finite family $\{f_1, \ldots, f_{p+q}\}\subset \mathcal{G}^+$ of $C^{1+\alpha}$ diffeomorphisms on $M$ and a Markov partition $\mathcal{R}=\{R_1, \ldots, R_{p+q}\}$ for which the conditions $(A0)$, $(A1)$, $(A2)$ and $(A3)$ hold. Moreover, suppose that $\mathcal{G}^+$ has weak-cycle property. We will prove that the semi-group action $\mathcal{G}^+$ is ergodic with respect to the Lebesgue measure.
We take $\mathcal{A}$ the set of points $x \in M$ that belong to the closure of any hyperbolic cylinder $C^n$ for infinitely many values $n_1(x)< n_2(x) < \ldots< n_k(x) < \ldots$ of $n$.
Clearly $\mathcal{H}\subseteq \mathcal{A}$. In particular, each $x \in \mathcal{A}$ has infinitely many hyperbolic times and hence $m(\mathcal{A})=1$.

Now, for each $n \geq 1$, we consider a covering $\mathcal{A}_n$ of $\mathcal{A}$ by pairwise disjoint measurable sets such that $A_n \in \mathcal{A}_n$ satisfying $C^m \subseteq A_n \subseteq \overline{C}^m$
for some $C^m \in \mathcal{C}_h^m$ and $m\geq n$. Indeed, we can take the family $\mathcal{A}_n$ which consists of cylinders $C^m \in \mathcal{C}_h^m$ with $m\geq n$ that are not contained in any
$C^k \in \mathcal{C}_h^k$ with $m>k \geq n$.
\begin{proposition}
Suppose that $B \subset M$ is a measurable $\mathcal{G}^+$-forward invariant subset with positive Lebesgue measure. Then there exists an element $R_i$ of the topological partition $\mathcal{R}$ such that $m(R_i \cap B^c)=0$.
\end{proposition}
\begin{proof}
Let $B \subset M$ be a measurable $\mathcal{G}^+$-forward invariant subset with positive Lebesgue measure.
 The next claim is needed.
\begin{claim}
Given $\delta >0$, there exist $n\geq 1$ and a subset $\{A_{n,i}: i\in I\}$ of $\mathcal{A}_{n}$ such that
$$m(B\bigtriangleup \bigcup_{i\in I}A_{n,i})\leq \delta.$$
\end{claim}
Indeed, it is enough to apply the technique of [Lemma 3.11, \cite{OV}] to our setting. Take compact subsets $K \subset B$ and $K^{\prime}\subset B^c$ such that $$m(B\bigtriangleup K)\leq\frac{\delta}{3}, \  and \
m(B^c \bigtriangleup K^{\prime})\leq\frac{\delta}{3}.$$ We set $\rho = dist(K, K^{\prime})$. By Corollary 4.8, one has $diam (\mathcal{A}_n)\leq K_2 e^{\frac{-nc}{2}}< \rho$ provided that $n$ is large enough.
Since $\mathcal{A}$ is covered by $\mathcal{A}_n$, so there exists a family $\{A_{n,i}: i\in I\}\subset \mathcal{A}_n$ so that $m(K \setminus \bigcup_{i\in I}A_{n,i})\leq \frac{\delta}{3}.$
Since all $A_{i,n}, \ i \in I$, intersect $K$, they are disjoint from $K^{\prime}$.
Therefore,
$$m(B\bigtriangleup \bigcup_{i\in I}A_{n,i})\leq m(B\setminus K)+m(K \setminus \bigcup_{i\in I}A_{n,i})+
m(B^c \setminus K^{\prime})\leq \delta.$$
Now, according to the above claim and the approach of [Corollary 3.12, \cite{OV}] the following claim holds.
\begin{claim}
For each $\epsilon >0$, there exist $n\geq 1$ and $A_n \in \mathcal{A}_n$ such that
$$\frac{m(B\cap A_n)}{m(A_n)}>1-\epsilon.$$
\end{claim}
By Claim 5.3, there exists a sequence $A_{n}$ of measurable sets for which the following holds: $C^{m_{n}}\subset A_n \subset \bar{C}^{m_{n}}$, for some hyperbolic cylinder $C^{m_{n}}$ with $m_{n}\geq n$ and
$$\frac{m(A_n \cap B^{c})}{m(A_n)} \to 0, \ n \to \infty.$$
Then, according to the distortion Lemma 4.9, Corollary 4.10 and the assumption that $B$ is forward invariant, one has that
\begin{equation}\label{2}
\frac{m(f^{m_{n}}_{\omega({m_n})}(A_n)\cap B^{c})}{m(f^{m_{n}}_{\omega({m_n})}(A_n))}\to 0, \ when \
n \to \infty,
 \end{equation}
where $\omega({m_n})$ is the itinerary of the cylinder $C^{m_n}= C^{m_n}[\omega_0, \ldots, \omega_{m_n -1}]$
and $f^{m_{n}}_{\omega({m_n})}= f_{\omega_{m_n-1}} \circ \ldots \circ f_{\omega_0}$.\\
By Remark 4.6, $$R_{\omega_{m_n -1}}\subset f^{m_{n}}_{\omega(m_n)}(C^{m_{n}}) \subset f^{m_{n}}_{\omega(m_n)}(A_n),$$ for some $\omega_{m_n -1}\in \{1,\dots , p+q\}$. Fix any $j$ such that $\omega_{m_n -1}=j$ for infinitely many values of $n$. We know that $R_{j}$ is an open subset and therefore has positive $m$-measure. Then $(6)$ implies that $m(R_{j}\cap B^{c})=0$.
\end{proof}
Now, we will finish the proof of Theorem A. As you have seen in the previous proposition, for forward invariant set $B$ with $m(B) > 0$, there exists a cylinder $R_j$ satisfying $m(R_{j}\cap B^{c})=0$.
It is enough to show that $m(B)=1$.

Let us fix the cylinder $R_j$. According to weak-cycle property of $\mathcal{G}^+$, there exists a sequence
$\{g_n\}\subset \mathcal{G}^+$ such that
$$M \circeq \bigcup_{n\geq1}g_n(R_j).$$
Since $B$ is $\mathcal{G}^+$-forward invariant, $B^c$ is also $\mathcal{G}^+$-forward invariant. Hence, one has that $$g_n(R_j)\setminus B \subset g_n(R_j\setminus B), \ for \ all \ n\geq 1.$$
The Lebesgue measure $m$ is quasi-invariant for $C^1$- diffeomorphisms which imply that
$m(g_n(R_j \setminus B))=0$ and therefore, $m(g_n(R_j) \setminus B)=0$.
Since $\{g_n(R_j)\}$ is a countable family, $m(B)=1$ which terminates the proof.
\section{\bf{Transitivity criteria}}
In this section we obtain some results for transitivity of semi-group actions of homeomorphisms defined on a compact manifold $M$. We introduce some kinds of transitive property and finally we provide a strong form of transitivity that suffices to conclude ergodicity in our setting.
First, we need to formalize some notions.

Consider a finite family $\{f_1, \ldots, f_k\}$ of homeomorphisms defined on a compact manifold $M$. Let us take $\mathcal{G}^+$ the semi-group generated by these homeomorphisms. For given $x \in M$, the \emph{forward total orbit} of $x$ is defined by $$\mathcal{O}_{\mathcal{G}}^+(x)= \{h(x): h \in \mathcal{G}^+\}.$$
Analogously, the \emph{backward total orbit} of $x$ is defined by $$\mathcal{O}_{\mathcal{G}}^-(x)= \{h^{-1}(x): h \in \mathcal{G}^+\}.$$
We say that $\mathcal{G}^+$ acts \emph{minimally} on $M$ if any point has a dense forward total orbit.\\
For any sequence $\omega=(\omega_0, \omega_1, \ldots, \omega_n, \ldots) \in \Sigma_k^+$, we write
$$f_\omega^n(x)=f_{\omega_n}\circ f_\omega^{n-1}(x); \ n \in \mathbb{N} \ and \ f_\omega^0=Id.$$
The set $\{f_\omega^n(x): n \geq 0\}$ is called $\omega$-\emph{orbital branch} of $x$. We say that $x$ possesses
an \emph{orbital dense orbit} if there exists a sequence $\omega \in \Sigma_k^+$ such that $\omega$-orbital branch of $x$ is dense in $M$.

It is easy to see that if $\mathcal{G}^+$ acts minimally on $M$, then each $x$ has orbital dense orbit.\\
In a finitely generated semi-group action, if minimality of action preserves under small perturbation of generators then we say that $\mathcal{G}^+$ is \emph{robustly minimal}. Some examples of robustly minimal semi-group actions are already available, for instance see \cite{GH} and \cite{H2}. Let us mention that transitive semi-group actions have a weak form of dynamical irreducible property with respect to minimal semi-groups.

Here, we introduce three kinds of transitive property. The action of $\mathcal{G}^+$ is \emph{weak transitive}
if it possesses a point with a dense forward total orbit. We say that $\mathcal{G}^+$ is \emph{transitive} whenever $U$ and
$V$ are two open subsets of $M$, there exists $h \in \mathcal{G}^+$ such that $h(U) \cap V \neq \emptyset$.
Finally, $\mathcal{G}^+$ is \emph{strong transitive} if it admits a full measure subset $\widetilde{M} \subset M$
such that every point $x \in \widetilde{M}$ has a dense forward orbit.

In bellow, we illustrate the relationship between these different concepts of transitive property. Moreover, to obtain some results about the transitive properties of semi-group actions, we adapt some techniques from deterministic systems \cite{LP} to our setting.

A semi-group $\mathcal{G}^+$ is \emph{contractive} at point $x \in M$ whenever there exist a sequence
$\{ h_n \}\subset \mathcal{G}^+$ and open set $B$ containing $x$ for which
$diam(h_n(B)) \to 0$, as $n \to \infty$.
\begin{lemma}
Let $\{f_{1},\dots ,f_{k}\}$ be a finite family of homeomorphisms on a compact manifold $M$ and $\mathcal{G}^+$ be the semi-group generated by them. Suppose that  there exist $x\in M$ for which the total backward orbit and total forward orbit are dense in $M$. If the semi-group $\mathcal{G}^+$ is contractive at $x$, then the action of $\mathcal{G}^+$ is transitive, i.e. for each two open subsets $U$ and $V$ of $M$ there exists $h\in <\mathcal{G}>^{+}$ such that $h(U)\cup V\neq\emptyset$.\\
In particular, the point $x$ possesses an orbital branch which is dense in $M$.
\end{lemma}
\begin{proof}
Suppose that $U$ and $V$ are two open subsets of $M$. Since the total forward orbit of $x$ is dense in $M$, there exist
$k_{1},k_{2}\in \mathcal{G}^{+}$ such that $k_{1}(x)\in U$ and $k_{2}(x)\in V$.\\
According to the density of $\mathcal{O}^{-}_{\mathcal{G}}(x)$, we can choose two sequences $\{z_{i}\}\subset U$ and $\{h_{i}\}\subset \mathcal{G}^{+}$ that satisfying the following property:
$h_{i}(z_{i})=x$ and $z_{i}$ converges to $k_{1}(x)$ when $i$ goes to infinity and
therefore, $h_{i}(k_{1}(x))$ is contained in a neighborhood $B$ of $x$, for large enough $i$.\\
On the other hand, continuity of $k_{2}$ and this fact that $k_{2}(x)\in V$ imply that $k_{2}(B_{r}(x))\subset V$, for some small $r>0$.\\
Now, by contractibility of $\mathcal{G}^{+}$ at $x$, one can find $g\in <\mathcal{G}>^{+}$ such that $g(B)\subset B_{r}(x)$,
and therefore $k_{2}(g(h_{i}(k_{1}(x))))\in V.$ Let us take $h:= k_{2}\circ g\circ h_{i}\circ k_{1}$, then $h(U)\cap V\neq\emptyset$.\\
Finally, if we apply the above argument for a countable basis of $M$, the second statement follows immediately.
\end{proof}
\begin{remark}
Let us take $\mathcal{G}^+ \subset Diff^{1+\alpha}(M)$ with generators $g_1, \ldots, g_k$ and there exists
$\{f_1, \ldots, f_{p+q}\} \subset \mathcal{G}^+$ satisfies the conditions proposed in subsection 2.1. Then Corollary 4.8 ensures that the inverse semi-group $\mathcal{G}^-=<g_1^{-1}, \ldots, g_k^{-1}>$ is contractive at each point $x \in \mathcal{H}$. Hence, Lemma 6.2 implies that if $\mathcal{G}^+$ possesses a point $x \in \mathcal{H}$ with forward and backward dense orbit, then both semi-groups $\mathcal{G}^+$ and $\mathcal{G}^-$ are transitive.
\end{remark}
\begin{lemma}
Suppose that $\{f_{1}, \ldots ,f_{k}\}$ is a finite family of homeomorphisms on $M$ and $\mathcal{G}^{+}$ is the semi-group generated by these maps.
If $\mathcal{G}^{+}$ has any point with forward orbital dense orbit, then the set of points with total forward dense orbit is a residual subset of $M$.
\end{lemma}
\begin{proof}
Suppose that there exists a point $x\in M$ and $\omega \in \Sigma^{+}_{k}$ such that the $\omega$-orbital orbit of $x$, $\mathcal{O}^{+}_{\mathcal{G}}(\omega, x)=\{f^{j}_{\omega}(x): j\in \mathbb{N}\}$, is dense in $M$. Let us take a countable basis $\mathcal{V}= \{ V_j : j \in \mathbb{N}\}$ of $M$ and we set
$$\mathcal{A}_n=\{z \in M: \mathcal{O}_{\mathcal{G}}^+(z) \ is \ \frac{1}{n}-dense\}.$$
By density of $\mathcal{O}^{+}_{\mathcal{G}}(\omega, x)$ there exist sequences $\{x_{j_{l}}\}\subset \mathcal{O}^{+}_{\mathcal{G}}(\omega, x)$ and $k_{n,l}\subset \mathbb{N}$, for each $n,l\in \mathbb{N}$, for which the following holds:\\
$x_{j_{l}}\in V_{l}$ and the segment orbit $\{x_{j_{l}}=f^{j_{l}}_{\omega}(x), f^{j_{l}+1}_{\omega}(x), \ldots f^{j_{l}+k_{n,l}}_{\omega}(x)\}$ is $\frac{1}{2n}$-dense. By continuity, there exists $\{r_{n,l}>0\}_{n,l\in \mathbb{N}}$ such that
 $f^{i}_{\sigma^{j_{l}-1}\omega}(B_{r_{n,l}}(x_{j_{l}}))\subset B_{\frac{1}{2n}}(f^{i}_{\sigma^{j_{l}-1}\omega}(x_{j_{l}}))$ for
 $0\leq i\leq k_{n,l}$. This means that for each $y\in B_{r_{n,l}}(x_{j_{l}})$, the segment orbit $\{y, f_{\sigma^{j_{l}-1}\omega}(y), \ldots ,  f^{k_{n,l}}_{\sigma^{j_{l}-1}\omega}(y)\}$ is $\frac{1}{n}$-dense.\\
 Therefore, $\cup_{l \in \mathbb{N}}B_{r_{n,l}}(x_{j_{l}})$ is an open and dense subset. In particular,
 this proves that $\cap_{n \in \mathbb{N}}\cup_{l \in \mathbb{N}}B_{r_{n,l}}(x_{j_{l}})$ is a residual subset contained in $\cap_{n=1}^{\infty}\mathcal{A}_n$ which consists of the points with dense forward orbit. This terminates the proof of the lemma.
 \end{proof}
\begin{lemma}
Let $\mathcal{F}$ be a finite family of homeomorphism on $M$. If there exists a point $x\in M$ with forward orbital dense orbit, then
$\{x\in M: \overline{\mathcal{O}^{-}_{\mathcal{F}}(x)}=\overline{\mathcal{O}^{+}_{\mathcal{F}}(x)}=M\}$ is a residual subset of $M$.
\end{lemma}
\begin{proof}
From the previous lemma, $\{x\in M: \overline{\mathcal{O}^{+}_{\mathcal{F}}(x)}=M\}$ is a residual subset of $M$. Let us take $x\in M$
with forward orbital dense orbit, so there exists $\omega \in \Sigma^{+}_{n}$ for which $\mathcal{O}^{+}_{\mathcal{F}}(x,\omega)$ is
dense in $M$, i.e, the set $\{x, f_{\omega_{0}}(x), \dots ,f^{j}_{\omega}(x), \dots\}$ is dense in $M$. Let us take $x_{l}=f^{l}_{\omega}(x)$
and $\mathcal{A}_{n}=\{x\in M: \mathcal{O}^{-}_{\mathcal{F}}(x)\ \  is\ \  \frac{1}{n}-dense\}$. Since $\mathcal{O}^{+}_{\mathcal{F}}(x,\omega)$ is dense, there exists $\{k_{n}\}$ such that $\{x, \ldots ,f^{k_{n}}_{\omega}(x)\}$ is $\frac{1}{2n}$-dense. As $\{x, \ldots , f^{k_{n}}_{\omega}(x)\}\subset \mathcal{O}^{-}(x_{j})$ for all $j\geq k_n$, we get that $x_{j}\in \mathcal{A}_{n}$ for each $j\geq k_{n}$.\\
By continuity, one can find $r_{j}>0$ such that $f^{l}_{\omega}(B_{r_{j}}(x))\subset B_{\frac{1}{2n}}(f^{l}_{\omega}(x))$,
for all $0\leq l\leq k_{j}$. This proves that $\mathcal{O}^{-}_{\mathcal{F}}(y)$ is $\frac{1}{n}$-dense for all
$y\in f^{j}_{\omega}(B_{r_{j}}(x))$ and $j\geq k_{n}$. As $f^{j}_{\omega}(B_{r_{j}}(x))$ is an open neighborhood of $x_{j}$ and $\{x_{j}: j\geq k_{n}\}$ is dense,
then $\cup _{j\geq k_{n}}f^{j}_{\omega}(B_{r_{j}}(x))\subset \mathcal{A}_{n}$
 is an open and dense subset $M$. Therefore
 $\cap_{n\geq1}\cup_{j\geq k_{n}}f^{j}_{\omega}(B_{r_{j}}(x))\subset \cap_{n\geq1}\mathcal{A}_{n}$
 is a residual subset of $M$. Since $\cap _{n\geq1}\mathcal{A}_{n}$ consists of points with dense backward orbit,
 then $\{x\in M : \overline{\mathcal{O}^{-}_{\mathcal{F}}(x)}=M\}\cap \{x\in M : \overline{\mathcal{O}^{+}_{\mathcal{F}}(x)}=M\}$ is also residual. This terminates the proof of the lemma.

\end{proof}
Let us recall that the a finitely generated semi-group $\mathcal{G}^+$ satisfies the \emph{weak cycle} property if for each open set $B\subset M$, there exists a sequence $\{h_{i}\}\subset \mathcal{G}^{+}$ such that $M\circeq \cup^{\infty}_{i=1}h_{i}(B)$, this means that $m(M\setminus \cup h_{i}(B))=0$.
\begin{lemma}
A finitely generated semi-group $\mathcal{G}^+$ of homeomorphisms on $M$ has weak cycle property if and only if there exists a full measure subset $\tilde{M}\subset M$ such that for each point $x\in \tilde{M}$, the total backward orbit of $x$ is dense in $M$.
\end{lemma}
\begin{proof}
First, suppose that $\mathcal{G}^+$ has weak cycle property. Let us take a countable basis $\{V_{j}: j\geq 1\}$ of $M$. Weak cycle property implies that there exists a sequence $\{h_{ij}\}_{i,j\in \mathbb{N}}\subset \mathcal{G}^+$ so that
 $M\circeq \cup^{\infty}_{i=1}h_{ij}(V_{j})$, for each $j\in \mathbb{N}$.\\ Suppose that $F_{j}=M\setminus \cup_{i}h_{ij}(V_{j})$, then
$F_{j}$ has zero Lebesgue measure, hence $F=\cup_{j}F_{j}$ has also zero Lebesgue measure. We set $\tilde{M}=M\setminus F$ and we show that each $x\in \tilde{M}$ has a dense backward orbit.\\ Indeed, if $x\in \tilde{M}$ then $x$ is not belong to $F$, which implies that $x$ is not belong to $F_{j}$, for each $j$. This means that $x\in h_{ij}(V_{j})$ for each $i$. Therefore $h^{-1}_{ij}(x)\in V_{j}$ and hence the backward orbit of $x$ is dense in $M$.

Conversely, suppose that there exists a full measure subset $\tilde{M}\subset M$ so that each $x\in \tilde{M}$ has dense backward orbit. This means that for each open set $B$ there exists $h_{x}\in \mathcal{G}^{+}$ with $h^{-1}_{x}(x)\in B$. Therefore $x\in h_{x}(B)$. Then, the family $\{h_{x}(B): x\in \tilde{M}\}$ is an open covering of $\tilde{M}$ and hence there exists a countable subcover $\{h_{x_{i}}(B): i\in \mathbb{N}\}$. If we set $h_{i}:= h_{x_{i}}$, then we conclude that $M\circeq \cup_{i}h_{i}(B)$.
\end{proof}
\section{\bf{Examples}}
In this section, we exhibit some examples of finitely generated semigroups of $Diff^{1+\alpha}(M)$ satisfying our hypothesis in subsection 2.1.
 To start the construction, we recall some concepts of topology.

The \emph{standard} $m$-\emph{simplex} is the set
\begin{center}
$\triangle_m=\{x \in\mathbb{R}^{m+1}:x_i\geq0 \ and \ \sum_{i=0}^{m+1}x_i=1\}.$
\end{center}

A \emph{general} $m$-\emph{simplex} is a subset of $M$ diffeomorphic to the standard $m$-simplex $\triangle_m$ and a \emph{general} $n$-\emph{face} is a subset of $M$ diffeomorphic to the standard $n$-face.

A \emph{triangulation} of a compact $m$- manifold $M$ is a finite collection $\mathcal{T}$ of diffeomorphic images of $\triangle_m$ that cover $M$ and satisfying the following condition:
for any general $m$-simplices $S_i,S_j\in \mathcal{T}$, if $S_i\cap S_j\neq\emptyset$, their intersection is a $(m-1)$-face of both $S_i$ and $S_j$.
\\
The \emph{barycentric subdivision} of an $m$-dimensional simplex $S$ consists of $(m+1)!$ simplices. Each piece, with vertices $v_0,\ldots, v_m$ can be associated with a permutation $p_0,\ldots, p_m$ of the vertices of $S$, in such a way that each vertex $v_i$ is the barycenter of the points $p_0,\ldots, p_i$. Barycentric subdivision is an important tool which is used to obtain finer simplicial complexes.

Now, we start by considering any smooth Riemannian manifold $M$ that supports orbital expanding or orbital non-uniformly expanding finitely generated semigroups of $Diff^{1+\alpha}(M)$.
We are going to introduce a finite family of $ Diff^{1+\alpha}(M)$
 that admits a Markov partition satisfying conditions $(A_0), (A_1), (A_2)$ and $(A_3)$.

It is a well-known fact that $M$ possesses a triangulation $\mathcal{T}=\{S_1, \ldots, S_k\}$ , where $S_i, \  i=1, \ldots, k$, are diffeomorphic images of $\triangle_m$ and cover $M$.
Suppose that $g_i, \ i=1, \ldots, k$, are diffemorphisms which map $\triangle_m$ to $S_i$.
\\
We divide $\triangle_m$ by barycentric subdivision to smaller $m$- dimensional $m$-simplices $\triangle_{m,1}, \ldots, \triangle_{m,l}$. Take affine maps $h_j, j=1, \ldots, l$, which map $\triangle_{m,j}$ to $\triangle_m$.
\\
Clearly, each $h_j$ is an expanding map. Now, we set $f_{i,j}:= g_i \circ h_j\circ g_i^{-1}$ and extend it to a $C^{1+\alpha}$-diffeomorphism on $M$. Also, take $S_{i.j}:= g_i(\triangle_{m,j})$, then the interiors of $S_{i,j}$, $i=1, \ldots, k$ , $j=1, \ldots, l$ exhibit a Markove partition, that is
if we take $R_{i,j}:= int(S_{i,j})$, then $(R_{i,j},f_{i,j})$ satisfying conditions $(A0), \ldots, (A3)$, with $p=kl$ and $q=0$. Clearly, the semigroup generated by $\{f_{i.j}\}$ is an orbital expanding finitely generated $C^{1+\alpha}$ semi-group.

Now, we will give some changes in the subdivision corresponding to one of generalized $m$-simplices that allows us to exhibit a finitely generated semi-group of $C^{1+\alpha}$-diffeomorphisms on any smooth manifold that is orbital non-uniformly expanding.

Again, we start with a triangulation $\mathcal{T}$ on a smooth compact Riemannian manifold $M$ and choose a generalized $m$-simplex $S^*\in \mathcal{T}$. Let $\{\triangle_{m,1}, \ldots, \triangle_{m,l}\}$ be a subdivision of standard $m$-simplex $\triangle_m$ in the following sense: there exists a sub-simplex $\triangle_{m,j}$ which admits the entire of an $(m-1)$-face of $\triangle_m$ as a face but other faces are strictly smaller than the faces of $\triangle_m$. Moreover, other sub-simplicies are smaller than $\triangle_m$ in all direction.

We take $h_*$ an affine transformation which maps $\triangle_{m,j}$ to $\triangle_m$. Also, suppose that $h_i$, $i=1, \ldots, l$, $i\neq j$, are affine transformations which map $\triangle_{m,i}$ to $\triangle_m$.
\\
Let $g$ be a diffeomorphism which takes  $\triangle_m$ diffeomorphically to $S^*$.
\\
We set $S_j^*:= g(\triangle_{m,j})$, $R_j^*:= int(S_j^*)$ and $f_j^*:= g \circ h^*\circ g^{-1} \mid_{S^*_j}$.
Also, we take $S_i^*:= g(\triangle_{m,i})$, $R_i^*:= int(S_i^*)$ and $f_i^*:= g \circ h_i\circ g^{-1} \mid_{S_i^*}$, for $i=1, \ldots, l$, with $i\neq j$.
\\
Clearly, this subdivision is not barycentric. Other elements of $\mathcal{T}$ admit barycentric subdivisions with the affine transformations which map them to the whole of standard $m$-simplex, analogous to the previous example.
It is easy to see that this partition is also Markov and  satisfies the conditions $(A0)$, $(A1)$, $(A2)$ and $(A3)$ with $q=1$.
\section*{Acknowledgments} We are grateful to M. Nassiri, A. Fakhari and M.Saleh for many useful conversations and valuable suggestions.

\end{document}